\documentclass[11pt]{article}

\usepackage{amsfonts,amssymb,amsmath,amsthm,epsfig,euscript}
\usepackage{epic, eepic, graphics}
\usepackage{graphicx}

\newcommand{\red}{\mathrm{red}}
\newcommand{\des}{\mathrm{des}}

\newtheorem{theorem}{Theorem}

\newtheorem{lemma}[theorem]{Lemma}

\long\def\symbolfootnote[#1]#2{\begingroup
\def\thefootnote{\fnsymbol{footnote}}\footnote[#1]{#2}\endgroup}

\newcommand\p{\circle*{0.3}}

\newcommand{\tpt}{$(\mathbf{2+2})$}

\DeclareMathOperator{\asc}{\mathrm{asc}}

\DeclareMathOperator{\run}{\mathrm{run}}
\DeclareMathOperator{\last}{\mathrm{last}}

\title{$p$-Ascent Sequences}

\author{
Sergey Kitaev \\
\small University of Strathclyde\\[-0.8ex]
\small Livingstone Tower, 26 Richmond Street\\[-0.8ex]
\small Glasgow G1 1XH, United Kingdom\\[-0.8ex]
\small \texttt{sergey.kitaev@cis.strath.ac.uk}
 \and
Jeffrey B. Remmel\footnote{Partially supported by NSF grant DMS 0654060.} \\
\small Department of Mathematics\\[-0.8ex]
\small University of California, San Diego\\[-0.8ex]
\small La Jolla, CA 92093-0112. USA\\[-0.8ex]
\small \texttt{remmel@math.ucsd.edu}
}

\date{\small Submitted: Date 1;  Accepted: Date 2;
 Published: Date 3.\\
\small MR Subject Classifications: 05A15}

\begin{document}
\maketitle

\begin{abstract}
A sequence $(a_1, \ldots, a_n)$ of nonnegative integers is an 
{\em ascent sequence} if $a_0 =0$ and for all $i \geq 2$, $a_i$ is at most 1 plus the number of ascents in $(a_1, \ldots, a_{i-1})$.  
Ascent sequences were introduced by Bousquet-M\'elou, Claesson, Dukes, 
and Kitaev in \cite{BCDK}, who showed that these sequences of length $n$ are in 1-to-1 correspondence with 
\tpt-free posets of 
size $n$, which, in turn, are in 1-to-1 correspondence with interval orders of size $n$. Ascent sequences are also in bijection with several other 
classes of combinatorial objects including the set of upper triangular matrices 
with nonnegative integer entries such that no row or column contains all zeros, permutations 
that avoid a certain mesh pattern, and the set of Stoimenow matchings. 

In this paper, we introduce a generalization of ascent sequences, 
which we call {\em $p$-ascent sequences}, where $p \geq 1$. 
A sequence $(a_1, \ldots, a_n)$ of nonnegative integers is a 
$p$-ascent sequence if $a_0 =0$ and for all $i \geq 2$, $a_i$ is at most $p$ plus the number of ascents in $(a_1, \ldots, a_{i-1})$. Thus, in our terminology, ascent sequences are 1-ascent sequences. We generalize 
a result of the authors in \cite{KR} by 
enumerating $p$-ascent sequences with respect to the number of $0$s. We also generalize a result of Dukes, Kitaev, Remmel, and 
Steingr\'{\i}msson in \cite{DKRS} by finding the generating function 
for the number of $p$-ascent sequences which have no consecutive repeated 
elements. Finally, we initiate the study of pattern-avoiding $p$-ascent sequences.

\end{abstract}

\section{Introduction}

\subsection{Ascent sequences}

Ascent sequences were introduced by Bousquet-M\'elou, Claesson, Dukes, 
and Kitaev in \cite{BCDK}, who showed that these sequences of length $n$ are in 1-to-1 correspondence with \tpt-{\em free posets} of 
size $n$. 

Let $\mathbb{N} =\{0,1, \ldots, \}$ denote the natural 
numbers and $\mathbb{N}^*$ denote the set of all words over 
$\mathbb{N}$. A sequence $(a_1,\dots , a_n ) \in \mathbb{N}^n$ is an
{\em ascent sequence of length $n$} if and only if it satisfies $a_1=0$ and
$a_i \in [0,1+\asc(a_1,\dots , a_{i-1})]$ for all $2\leq i \leq n$.
Here, for any integer sequence $(a_1,\dots , a_i)$, the number of
{\em{ascents}} of this sequence is
\begin{equation*}
  \asc(a_1,\dots , a_{i}) = |\{j: a_j<a_{j+1}; 1\leq j <i\}|.
\end{equation*}
For instance, (0, 1, 0, 2, 3, 1, 0, 0, 2) is an ascent sequence 
which has 4 ascents. 
We let $Asc$ denote the set of all ascent sequences, where we assume that the empty
word is also an ascent sequence. For any $n \geq 1$, we let 
 $Asc_n$ denote the set of all ascent sequences of length $n$.  
If $a = (a_1, \ldots, a_n) \in Asc_n$, we let $|a| =n$ be the length  
of $a$, $\sum a = a_1 + \cdots +a_n$ equal the sum of the values 
of $a$,  $|a|_0$ denote the number of occurrences of 
$0$ in $a$, and $\last(a) =a_n$ denote the last letter of $a$.  
We say that $a = (a_1, \ldots, a_n) \in Asc_n$ is an {\em up-down} ascent 
sequence if $a_1 < a_2 > a_3 < a_4 > \cdots$. 
That is, $a = (a_1, \ldots, a_n) \in Asc_n$ is an up-down  ascent 
sequence if $a_i < a_{i+1}$ whenever $i$ is odd, and 
 $a_i > a_{i+1}$ whenever $i$ is even. Throughout this 
paper, we will often identify a sequence 
$(a_1, \ldots , a_n)$ in $\mathbb{N}^n$ with the word 
$a_1 \ldots a_n$.  Thus, instead of writing, say, $(0,0,0)$, we will 
simply write 000, or $0^3$.

We note that there has been considerable work on ascent sequences 
in recent years, see, for example, \cite{BCDK,DKRS,DS,KR}. Ascent sequences are important because 
they are in bijection with several other interesting combinatorial 
objects. To be more precise, it follows from the work of \cite{BCDK,DP,cdk} 
that there are natural bijections between $Asc_n$ and 
the following four classes 
of combinatorial objects. 

\begin{itemize}
\item The set  of \tpt-free posets of size $n$. Here we consider two posets to be equal if they are isomorphic, and an unlabeled poset is said to be \tpt-free if it does not
contain an induced subposet that is isomorphic to \tpt, the union
of two disjoint 2-element chains.  \tpt-free posets are known to be in 1-to-1 correspondence with celebrated {\em interval orders}.
\item The set $M_{n}$ of upper triangular matrices of nonnegative
integers such that
no row or column contains all zero entries, and
the sum of the entries is $n$.
\item The set  $R_n$ of permutations of $[n]=\{1,\ldots,n\}$, where in
each occurrence of the pattern 231, either the letters corresponding
to the 2 and the 3 are nonadjacent, or else the letters corresponding
to the 2 and the 1 are nonadjacent in value. Here, a word contains an occurrence of the pattern 231 if it contains a subsequence of length 3 that is order-isomorphic to 231; see \cite{Kit} for a comprehensive introduction to the theory of patterns in permutations and words.
\item The set $Mch_n$ of Stoimenow matchings on $[2n]$. A \emph{matching} of the set $[2n]=\{1,2,\ldots,2n\}$ is a partition
of $[2n]$ into subsets of size 2, each of which is called an
\emph{arc}.  The smaller number in an arc is its {\em opener}, and the
larger one is its {\em closer}.  A matching is said to be
\emph{Stoimenow} if it has no pair of arcs $\{a<b\}$ and $\{c<d\}$ 
that satisfy one (or both) of the following
conditions: (a) $a=c+1$ and $b<d$ and (b) $a<c$ and $b=d+1$. 
In other words, a Stoimenow matching has no pair of arcs such that one
is nested within the other and either the openers or the closers of the two
arcs differ by 1.
\end{itemize}

Also, Remmel \cite{Rem} showed that there is an interesting 
connection between the {\em Genocchi numbers} 
$G_{2n}$ and the {\em median Genocchi numbers} $H_{2n-1}$ and up-down 
ascent sequences.  In particular, Remmel showed 
that $G_{2n}$ is the number of 
up-down ascent sequences of length $2n-1$, $H_{2n-1}$ is 
the number of up-down ascent sequences of length $2n-2$, and 
that up-down ascent sequences can be used to give a natural 
combinatorial interpretation of the $q$-Genocchi numbers 
of Zeng and Zhou \cite{ZZ}.

Let $p_n$ be the number of \tpt-free posets on $n$ elements or, equivalently, 
the number of ascent sequences of length $n$. 
Bousquet-M\'elou et al.~\cite{BCDK} showed that the generating
function for the number $p_n$ of \tpt-free posets on $n$ elements is
\begin{equation}\label{gf}
P(t)= \sum_{n\geq 0} p_n  t^n=\sum_{n\ge 0} \prod_{i=1}^{n} \left(
1-(1-t)^i\right).  
\end{equation}
In fact, Bousquet-M\'elou et al.~\cite{BCDK} studied a
more general generating function 
$$F(t,u,v) = \sum_{w \in Asc} t^{|w|} u^{\asc(w)} v^{\last(w)}$$
and found an explicit form for such a generating function. 
Kitaev and Remmel \cite{KR} studied a refined version of 
this generating function. That is, they found 
an explicit formula for the generating function 
$$G(t,u,v,z,x):=\sum_{w \in
Asc}t^{|w|}u^{\asc(w)}v^{\last(w)}z^{|w|_0}x^{\run(w)},$$
where for any ascent sequence 
$w$, $\run(w) = 0$ if $w = 0^n$ for some $n$, and 
$\run(w) =r$ if $w = 0^rxv$, where $x$ is a positive integer and $v$ is a word. 
Thus $\run(w)$ keeps track of the initial sequences of 0s that 
start out $w$ if $w$ does not consist of all zeros. 
Kitaev and Remmel were able to use their formula for 
$G(t,u,v,z,x)$ to prove that 
\begin{equation}\label{a-nice-res}
A(t,z):= \sum_{w \in Asc} t^{|w|} z^{|w|_0} = 
1+ \sum_{n \geq 0} \frac{zt}{(1-zt)^{n+1}} \prod_{i=1}^n (1-(1-t)^i).
\end{equation}

\subsection{$p$-ascent sequences}

In this paper, we introduce a generalization of ascent sequences, 
which we call {\em $p$-ascent sequences}, where $p \geq 1$. 
A sequence $(a_1, \ldots, a_n)$ of nonnegative integers is a 
{\em $p$-ascent sequence} if $a_0 =0$ and for all $i \geq 2$, $a_i$ is at most $p$ plus the number of ascents in $(a_1, \ldots, a_{i-1})$. Thus, in our terminology, ascent sequences are 1-ascent sequences.

We note that $p$-ascent sequences of length $n$ can be encoded in terms of (usual) ascent sequences of length $n+2p-2$. 
Indeed, it is easy to see that $(a_1, a_2, \ldots, a_n)$ is a $p$-ascent sequence if
and only if $(0,1,0,1,\ldots,0,1,a_1, a_2, \ldots, a_n)$ is an ascent sequence, where
there are $p-1$ 0s and $p-1$ 1s preceding the $a_1 =0$.  Thus, $p$-ascent sequences can be thought of as a subset of ascent sequences of special type, 
namely, those ascents sequences that start out 
with $(01)^{p-1}0$. 

The last observation allows to obtain a characterization of elements counted by
$p$-ascent sequences in \tpt-free posets, the set of restricted permutations $R_n$, the set of upper triangular
matrices $M_n$, and the set of Stoimenow matchings $Mch_n$ whenever we can characterize the images of
ascent sequences whose corresponding words start with $(01)^{p-1}0$.  We do not get into much detail here providing just one example; we leave the other cases to the interested reader to explore. The \tpt-free posets corresponding to $p$-ascent sequences are \tpt-free posets on $n+2p-2$ elements with the following property. Right before the last $2p-1$ steps in decomposition of such posets (the decomposition is described in~\cite{BCDK}; we do not provide its details here due to space concerns), one obtains the poset with $p$ minimum elements and the other $p-1$ elements forming the pattern of the poset in Figure~\ref{poset-p-asc-seq} corresponding to the case $p=5$.  Of course, it would be interesting to give a direct characterization of such posets (e.g., in terms of forbidden sub-posets) but we were not able to succeed with that. 

\begin{figure}[h]
\begin{center}
\begin{picture}(8,8)

\put(0,0){

\put(0,0){\p} 
\put(2,0){\p} 
\put(4,0){\p} 
\put(6,0){\p} 
\put(8,0){\p} 
\put(0,2){\p} 
\put(2,4){\p} 
\put(4,6){\p} 
\put(6,8){\p} 

\put(0,0){\line(0,1){2}}
\put(0,0){\line(2,4){2}}
\put(0,0){\line(4,6){4}}
\put(0,0){\line(6,8){6}}

\put(2,0){\line(0,1){4}}
\put(2,0){\line(2,6){2}}
\put(2,0){\line(4,8){4}}

\put(4,0){\line(0,1){6}}
\put(4,0){\line(2,8){2}}

\put(6,0){\line(0,1){8}}

}

\end{picture}
\caption{Type of poset obtained right before the last $2p-1$ steps in decomposition of the \tpt-free poset corresponding to a $p$-ascent sequence.} \label{poset-p-asc-seq}
\end{center}
\end{figure}
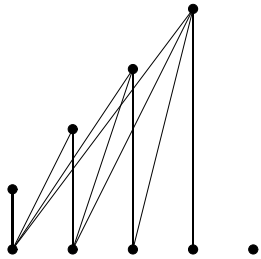

The main goal of this paper is to generalize 
the results of \cite{KR} to $p$-ascent sequences. 
That is, let 
$Asc(p)$ denote the set of $p$-ascent sequences, where, again, 
we consider the empty word to be a $p$-ascent sequence for 
any $p \geq 1$.  Thus, the set of ascent sequences $Asc$ is $Asc(1)$ in our terminology. 
First, we shall study the generating functions 
\begin{equation}
G^{(p)}(t,u,v,z,x):=\sum_{w \in
Asc(p)}t^{|w|}u^{\asc(w)}v^{\last(w)}z^{|w|_0}x^{\run(w)}.
\end{equation}
We shall find an explict formula for $G^{(p)}(t,u,v,z,x)$ for 
any $p \geq 1$ (see Section~\ref{sec2}) and then we shall use that formula to prove 
that 
\begin{equation}
A^{(p)}(t,z):= \sum_{w \in Asc(p)} t^{|w|} z^{|w|_0} = 
1+ \sum_{n \geq 0} \binom{p+n-1}{n} 
\frac{zt}{(1-zt)^{n+1}} \prod_{i=1}^n (1-(1-t)^i).
\end{equation}

 Duncan and Steingr\'{i}msson \cite{DS} introduced the study of 
pattern avoidance in ascent sequences.  We initiate a similar 
study for $p$-ascent sequences.   
Given a word $w = w_1 \ldots w_n \in \mathbb{N}^*$, we let 
$\red(w)$ denote the word that is obtained from $w$ by replacing 
each copy of the $i$-th smallest element in $w$ by $i-1$. For 
example, $\red(238543623) =015321401$.  Then we say that a 
word $u = u_1 \ldots u_j$ {\em occurs} in $w$ if there 
exist $1 \leq i_1 < \cdots < i_j \leq n$ such that 
$\red(w_{i_1}w_{i_2}\ldots w_{i_j}) = u$. We say that $w$ 
{\em avoids} $u$ if $u$ does not occur in $w$.

For any word $u \in \mathbb{N}^*$ such that $\red(u) =u$, we 
let $a_{n,p,u}$ denote the number of $p$-ascent sequences $a$ of length 
$n$ 
avoiding $u$, and  $r_{n,p,u}$ 
denote the number of sequences counted by $a_{n,p,u}$ with no equal consecutive letters, that is, $r_{n,p,u}$  is the number of {\em primitive sequences} counted by $a_{n,p,u}$. We prove a number of 
results about $a_{n,p,u}$ and $r_{n,p,u}$.  For example, 
we will show that for all $p \geq 1$, 
\begin{eqnarray*}
r_{n,p,10} &=& \binom{p+n-2}{n-1} \ \mbox{and} \\
a_{n,p,10} &=& \sum_{s=0}^{n-1} \binom{n-1}{s}\binom{p+s-1}{s}.
\end{eqnarray*}

This paper is organized as follows. In Section~\ref{sec2}, we shall 
find an explicit formula for $G^{(p)}(t,u,v,z,x)$. Unfortunately, 
we can not directly set $u=1$ in that formula so that in 
Section~\ref{lenzeros}, 
we shall find a formula for $G^{(p)}(t,1,1,1,x)$ via an alternative 
proof. This formula will also allow us to find an explicit formula 
for the generating function for the number of primitive $p$-ascent sequences.
Finally, in Section~\ref{pattern-avoidance}, we shall study $a_{n,p,u}$ and $r_{n,p,u}$ 
for certain patterns $u$ of lengths 2 and 3.

\section{Main results}\label{mainresults}\label{sec2}

For $r \geq 1$, let $G^{(p)}_r(t,u,v,z)$ denote the coefficient of $x^r$
in $G^{(p)}(t,u,v,z,x)$.  Thus $G^{(p)}_r(t,u,v,z)$ is the generating function
of those $p$-ascent sequences that begin with $r \geq 1$ 0s followed by
some element between 1 and $p$. We let $G^{(p,r)}_{a,\ell,m,n}$ denote the number of $p$-ascent sequences of
length $n$, which begin with $r$ 0s followed by some element 
between 1 and $p$, have $a$ ascents,
 last letter $\ell$, and a total of $m$ zeros. We then let
\begin{equation}
G^{(p)}_r(t,u,v,z)=\sum_{a,\ell,m\geq 0,\ n \geq r+1}
G^{(p,r)}_{a,\ell,m,n} t^n u^a v^{\ell} z^m.
\end{equation}

Clearly, since the sequences of the form  $0^n$ for some $n$ have no ascents and no
initial run of 0s (by definition), we have that the generating
function for such sequences is $$1+tz +
(tz)^2+\cdots=\frac{1}{1-tz},$$ where 1 corresponds to the empty
word. Thus, we have the following relation between $G^{(p)}$ and $G^{(p)}_r$:
\begin{equation}\label{solution}
G^{(p)}(t,u,v,x,z)=\frac{1}{1-tz}+\sum_{r\geq 1}x^r G^{(p)}_r(t,u,v,z).
\end{equation}

\begin{lemma}\label{lem2} For $r \geq 1$, the generating function 
$G^{(p)}_r(t,u,v,z)$ satisfies
\begin{multline}\label{main}
(v-1-tv(1-u))G^{(p)}_r(t,u,v,z)=\\
t^{r+1}z^ruv(v^p-1)+t((v-1)z-v)G^{(p)}_r(t,u,1,z)+tuv^{p+1}G_r(t,uv,1,z).
\end{multline}\end{lemma}

\begin{proof}

Our proof follows the same steps as the proof of the $p=1$ case 
of the result that was provided in \cite{KR}. Fix $r \geq 1$.  Let
$x'=(x_1,\ldots,x_{n-1})$ be an ascent sequence
beginning with $r$ 0s followed by a nonzero element, with $a$ ascents and $m$ zeros,
where $x_{n-1}=\ell$. Then $x=(x_1,
  \dots, x_{n-1}, i)$ is an ascent sequence if and only if $i\in [0,
    a+p]$. Clearly, $x$ also begins with $r$ 0s followed by a nonzero element. Now,
if $i=0$, the sequence $x$ has $a$
    ascents and $m+1$ zeros. If $1\leq i\leq \ell$, $x$ has
    $a$ ascents and $m$ zeros. Finally if $i \in [\ell+1,a+p]$, then
$x$ has $a+1$ ascents and
    $m$ zeros. Counting the sequences $0\ldots 0q$ with $r$ 0s and 
$1 \leq q \leq p$
    separately, we have

\begin{eqnarray*}
G^{(p)}_r(t,u,v,z) &=& t^{r+1}uvz^r\frac{v^p-1}{v-1}+\\
&&\sum_{\stackrel{a,\ell,m\geq
0}{n \geq r+1}}G^{(p,r)}_{a,\ell,m,n}t^{n+1}\left(u^av^0z^{m+1}+\sum_{i=1}^{\ell}u^av^iz^m+\sum_{i=\ell+1}^{a+p}u^{a+1}v^iz^m\right)\\
&=& t^{r+1}uvz^r\frac{v^p-1}{v-1}+t\sum_{\stackrel{a,\ell,m\geq
0}{n \geq r+1}}G^{(p,r)}_{a,\ell,m,n}t^nu^az^m\left(z+\frac{v^{\ell+1}-v}{v-1}+u\frac{v^{a+p+1}-v^{\ell+1}}{v-1}\right)\\
&=&
t^{r+1}uvz^r\frac{v^p-1}{v-1}+tzG^{(p)}_r(t,u,1,z)+\\
&&tv\frac{G^{(p)}_r(t,u,v,z)-G_r(t,u,1,z)}{v-1}+tuv\frac{v^pG_r(t,uv,1,z)-G_r(t,u,v,z)}{v-1}.
\end{eqnarray*}

The result follows.
\end{proof}

Next, just like in the proof of the $p=1$ case in \cite{KR}, we use
the kernel method to proceed. Setting
$(v-1-tv(1-u)) =0$ and solving for $v$, we obtain that
the substitution $v=1/(1+t(u-1))$ will eliminate the left-hand side of
(\ref{main}). We can then solve for $G^{(p)}_r(t,u,1,z)$ to
obtain  that
\begin{equation}\label{relation1}
(1+zt(u-1))G^{(p)}_r(t,u,1,z)= \frac{t^rz^ru(1-\delta_1^p)}{\delta_1^p} + 
\frac{u}{\delta_1^p}G_r^{(p)}(t,\frac{u}{\delta_1},1,z).
\end{equation}
Setting $\gamma_1 = 1+zt(u-1)$, we see that 
\begin{equation}\label{relation2}
G^{(p)}_r(t,u,1,z)= \frac{t^rz^ru}{\gamma_1\delta_1^p}(1-\delta_1^p) + 
\frac{u}{\gamma_1\delta_1^p}G_r^{(p)}(t,\frac{u}{\delta_1},1,z)
\end{equation}
where $\delta_1 = 1+t(u-1)$.

Next we define
\begin{eqnarray}
\delta_k &:=& u-(1-t)^k(u-1) \ \mbox{and} \\
\gamma_k &:=& u-(1-zt)(1-t)^{k-1}(u-1)
\end{eqnarray}
for $k \geq 1$. We also set $\delta_0 = \gamma_0 =1$.
Observe that
$\delta_1 = u -(1-t)(u-1) = 1+t(u-1)$ and
$\gamma_1 = u -(1-zt)(u-1) = 1+zt(u-1)$. 

For any function of $f(u)$, we shall write $f(u)|_{u = \frac{u}{\delta_k}}$
for $f(u/\delta_k)$. It is then easy to check that
\begin{enumerate}
\item $ \displaystyle (u-1)|_{u = \frac{u}{\delta_k}} = \frac{(1-t)^k(u-1)}{\delta_k}$,

\item $ \displaystyle \delta_s|_{u = \frac{u}{\delta_k}} = \frac{\delta_{s+k}}{\delta_k}$,

\item $ \displaystyle \gamma_s|_{u = \frac{u}{\delta_k}} = \frac{\gamma_{s+k}}{\delta_k}$, and

\item $ \displaystyle \frac{u}{\delta_s}|_{u = \frac{u}{\delta_k}} = \frac{u}{\delta_{s+k}}$.
\end{enumerate}

Using these relations, one can iterate the recursion (\ref{relation2}). 
For example, 
\begin{eqnarray*}
G^{(p)}_r(t,u,1,z) &=& \frac{t^rz^r u(1-\delta_1^p)}{\gamma_1\delta_1^p} +\\
&& \frac{u}{\delta_1^p} \left(  
\frac{t^rz^r \frac{u}{\delta_1}\left(1-\frac{\gamma_2}{\delta_1}\frac{\delta_2^p}{\delta_1^p}\right)}{
\frac{\gamma_2}{\delta_1} \frac{\delta_2^p}{\delta_1^p}} + 
\frac{\frac{u}{\delta_1}}{\frac{\delta_2^p}{\delta_1^p}}G^{(p)}_r(t,\frac{u}{\delta_2},1,z)\right)\\
&=&  \frac{t^{r}z^ru(1-\delta_1^p)}{\gamma_1\delta_1^p) } + 
 \frac{t^{r}z^ru^2\left(1-\frac{\delta_2^p}{\delta_1^p}\right)}{\gamma_1\gamma_2
\delta_2^p} +  
\frac{u^2}{\gamma_1 \gamma_2 \delta_2^p}G^{(p)}_r(t,\frac{u}{\delta_2},1,z).
\end{eqnarray*}
In general,
\begin{eqnarray*}
\frac{u^k}{\gamma_1 \cdots \gamma_k \delta_k^p}
G^{(p)}_r(t,\frac{u}{\delta_k},1,z) &=& 
\frac{u^k}{\gamma_1 \ldots \gamma_k\delta_k^p}\left(  
\frac{t^{r}z^r\frac{u}{\delta_k}\left(1-\frac{\delta_{k+1}^p}{\delta_k^p}\right)}{\frac{\gamma_{k+1}}{\delta_k}\frac{\delta_{k+1}^p}{\delta_k^p}} + 
\frac{\frac{u}{\delta_k}}{\frac{\gamma_{k+1}}{\delta_k}\frac{\delta_{k+1}^p}{\delta_k^p}}G^{(p)}_r(t,\frac{u}{\delta_{k+1}},1,z)\right)\\
&=&   
 \frac{t^{r}z^ru^{k+1}\left(1-\frac{\delta_{k+1}^p}{\delta_k^p}\right)}{\gamma_1\cdots \gamma_{k+1}\delta_{k+1}^p} + 
\frac{u^{k+1}}{\gamma_1 \cdots \gamma_{k+1}\delta_{k+1}^p}G^{(p)}_r(t,\frac{u}{\delta_{k+1}},1,z).
\end{eqnarray*}
Thus, by iterating recursion (\ref{relation2}), we can derive that  
\begin{equation}
G^{(p)}_r(t,u,1,z) = \frac{t^{r}z^ru(1-\delta_1^p)}{\gamma_1\delta_1^p } 
+ \sum_{k=2}^\infty \frac{t^{r}z^ru^{k}\left(1-\frac{\delta_{k}^p}{\delta_{k-1}^p}\right)}{\gamma_1\cdots \gamma_{k} \delta_k^p}.
\end{equation}
Note that since $\delta_0 =1$, we can rewrite  
$\displaystyle \frac{t^{r+1}z^ru(1-\delta_1^p)}{\gamma_1\delta_1^p}$ as 
$\displaystyle  \frac{t^{r}z^ru(\delta_0^p-\delta_1^p)}{\gamma_1\delta_0^p\delta_1^p }$
and we can rewrite 
$\displaystyle \frac{t^{r}z^ru^{k}\left(1-\frac{\delta_{k}^p}{\delta_{k-1}^p}\right)}{\gamma_1\cdots \gamma_{k} \delta_k^p}$ as 
$\displaystyle  \frac{t^{r}z^ru(\delta_{k-1}^p-\delta_k^p)}{\gamma_1\cdots \gamma_k 
\delta_{k-1}^p\delta_k^p }$. 
Thus we have proved the following theorem. 
\begin{theorem}
\begin{equation}\label{GR1}
G^{(p)}_r(t,u,1,z) = 
\sum_{k=1}^\infty \frac{t^{r}z^ru^{k}(\delta_{k-1}^p-\delta_{k}^p)}{\gamma_1\cdots \gamma_{k} \delta_{k-1}^p\delta_k^p}.
\end{equation}
\end{theorem}

We have used Mathematica to compute that
\begin{eqnarray*}
&&G^{(2)}_1(t,u,1,z) = 2u z t^2+\left(3zu^2+(3z+2z^2)u\right) t^3 +\\
&&\left(4zu^3+(13z+9z^2)u^2+(4z+3z^2+2z^3)u\right) t^4+\\
&&\left(5zu^4+(39z+28z^2)u^3+(35z+34z^2+15z^3)u^2+(5z+4z^2+3z^3+2z^4)u\right) t^5+O[t]^6.
\end{eqnarray*}

\begin{eqnarray*}
&&G^{(3)}_1(t,u,1,z) = 3u z t^2+\left(6zu^2+(6z+3z^2)u\right) t^3 + \\
&&\left(10zu^3+(34z+18z^2)u^2+(10z+6z^2+3z^3)u\right) t^4+\\
&&\left(15zu^4+(125z+70z^2)u^3+(115z+88z^2+30z^3)u^2+(15z+10z^2+6z^3+3z^4)u\right) t^5+\\
&&O[t]^6.
\end{eqnarray*}

\begin{eqnarray*}
&&G^{(4)}_1(t,u,1,z) = 4u z t^2+\left(10zu^2+(10z+4z^2)u\right) t^3 + \\
&&\left(20zu^3+(70z+30z^2)u^2+(20z+10z^2+4z^3)u\right) t^4+\\
&&\left(35zu^4+(305z+140z^2)u^3+(285z+180z^2+50z^3)u^2+(35z+20z^2+10z^3+4z^4)u\right) t^5+\\
&&O[t]^6.
\end{eqnarray*}

For example, the coefficient of $t^3$ in $G^{(2)}_1(t,u,1,z)$, which is $3zu^2+(3z+2z^2)u$, makes sense since there are eight 2-ascent sequences 
which start with 0 and are followed by a nonzero element, namely,
$$010,011,012, 013, 020,021, 022, \ \mbox{and}\  023,$$
of which three have one zero and two ascents, three have one zero and one ascent, and 
two have two zeros  and one ascent. Similarly, there a 19 3-ascent 
sequences of length 4 which have only one ascent, namely,
\begin{eqnarray*}
&&0111, 0110, 0100, 0222, 0221,0220, 0211, 0210, 0200. \\
&& 0333,0332,0331,0330,0322,0321,0320,0311,0310,\ \mbox{and}\  0300,
\end{eqnarray*}
three of which have three zeros, six of which have two zeros, and ten of which have 
one zero, and this is consistent with the term $(10z+6z^2+3z^3)ut^4$.

Note that we can rewrite (\ref{main}) as
\begin{equation}\label{main2}
G^{(p)}_r(t,u,v,z) = \frac{t^{r+1}z^ruv(v^p-1)}{v\delta_1-1} +
\frac{t(z(v-1)-v)}{v\delta_1-1}G^{(p)}_r(t,u,1,z) +
\frac{uv^{p+1}t}{v\delta_1-1}
G^{(p)}_r(t,uv,1,z).
\end{equation}

For $s \geq 1$, we let
\begin{eqnarray*}
\bar{\delta}_s &=& \delta_s|_{u =uv} = uv-(1-t)^s(uv-1) \ \mbox{and}\\
\bar{\gamma}_s &=& \gamma_s|_{u =uv} = uv-(1-zt)(1-t)^{s-1}(uv-1)
\end{eqnarray*}
and set $\bar{\delta}_0 = \bar{\gamma}_0 = 1$. Then using
(\ref{main2}) and (\ref{GR1}), we have the following theorem.

\begin{theorem}\label{Gr}
For all $r \geq 1$,
\begin{multline}
G^{(p)}_r(t,u,v,z) =\\ t^rz^r \left(\frac{tuv(v^p-1)}{v\delta_1-1} + \frac{t(z(v-1)-v)}{v\delta_1-1} \sum_{k \geq 1} 
\frac{(\delta_{k-1}^p -\delta_k^p)}{\gamma_1 \cdots \gamma_k 
\delta_{k-1}^p\delta_k^p} + \frac{tuv^{p+1}}{v\delta_1 -1}
\sum_{k \geq 1} 
\frac{(\bar{\delta}_{k-1}^p -\bar{\delta}_k^p)}{\bar{\gamma}_1 \cdots \bar{\gamma}_k 
\bar{\delta}_{k-1}^p\bar{\delta}_k^p}\right).
\end{multline}
\end{theorem}

We have used Mathematica to compute that
\begin{eqnarray*}
&&G^{(2)}_1(t,u,v,z) =  (u v z+u v^2 z)t^2 + (2 u v z+u v^2 z+u^2 v^2 z+2 u^2 v^3 z+2 u z^2)t^3+\\
&& \left( 3 u v z+3 u^2 v z+u v^2 z+5 u^2 v^2 z+5 u^2 v^3 z+u^3 v^3 z+3 u^3 v^4 z+3 u z^2+3 u^2 z^2+ \right. \\
&&\left. \ \ 2 u^2 v z^2+2 u^2 v^2 z^2+2 u^2 v^3 z^2+2 u z^3 \right)t^4 + \\
&& \left(4 u v z+13 u^2 v z+4 u^3 v z+u v^2 z+13 u^2 v^2 z+7 u^3 v^2 z+9 u^2 v^3 z+12 u^3 v^3 z+ \right. \\
&&\ \ 16 u^3 v^4 z+u^4 v^4 z+4 u^4 v^5 z+4 u z^2+13 u^2 z^2+4 u^3
z^2+9 u^2 v z^2+3 u^3 v z^2+\\
&&\ \ 7 u^2 v^2 z^2+5 u^3 v^2 z^2+5 u^2 v^3 z^2+7 u^3 v^3 z^2+
9 u^3 v^4 z^2+3 u z^3+9 u^2 z^3+\\
&&\left. \ \ 2 u^2 v z^3+2 u^2 v^2 z^3+2 u^2
v^3 z^3+2 u z^4\right)t^5 + O[t]^6.
\end{eqnarray*}

\begin{eqnarray*}
&&G^{(3)}_1(t,u,v,z) = (u v z+u v^2 z+u v^3 z)t^2+ \\
&&\left(3 u v z+2 u v^2 z+u^2 v^2 z+u v^3 z+2 u^2 v^3 z+3 u^2 v^4 z+3 u z^2\right)t^3+\\
&&\left( 6 u v z+6 u^2 v z+3 u v^2 z+9 u^2 v^2 z+u v^3 z+10 u^2 v^3 z+u^3 v^3 z+9 u^2 v^4 z+3 u^3 v^4 z+\right. \\
&&\left. \ \ 6 u^3 v^5 z+6 u z^2+6 u^2 z^2+3 u^2 v z^2+3 u^2 v^2
z^2+3 u^2 v^3 z^2+3 u^2 v^4 z^2+3 u z^3\right) t^4 + \\
&&\left( 10 u v z+34 u^2 v z+10 u^3 v z+4 u v^2 z+34 u^2 v^2 z+16 u^3 v^2 z+u v^3 z+28 u^2 v^3 z+25 u^3 v^3 z+\right. \\
&&\ \ 19 u^2 v^4 z+34 u^3 v^4 z+u^4 v^4 z+40 u^3
v^5 z+4 u^4 v^5 z+10 u^4 v^6 z+10 u z^2+34 u^2 z^2+ 10 u^3 z^2+\\
&& \ \ 18 u^2 v z^2+6 u^3 v z^2+15 u^2 v^2 z^2+9 u^3 v^2 z^2+12 u^2 v^3 z^2+12 u^3 v^3 z^2+9u^2 v^4 z^2+15 u^3 v^4 z^2+\\
&&\left. \ \ 18 u^3 v^5 z^2+6 u z^3+18 u^2 z^3+3 u^2 v z^3+3 u^2 v^2 z^3+3 u^2 v^3 z^3+3 u^2 v^4 z^3+3 u z^4\right) t^5 + O[t]^6.
\end{eqnarray*}
For example, the term $3 u v^2 z t^4$ that appears in 
$G^{(3)}_1(t,u,v,z)$ corresponds to the sequences 
0222, 0322, and 0332.  

 It is easy to see from Theorem \ref{Gr} that
\begin{equation}\label{rel16}
G^{(p)}_r(t,u,v,z) = t^{r-1}z^{r-1} G^{(p)}_1(t,u,v,z).
\end{equation}
The relation (\ref{rel16}) is also easy to see combinatorially since every ascent
sequence counted by $G^{(p)}_r(t,u,v,z)$ is of the form $0^{r-1}a$, where
$a$ is a $p$-ascent sequence counted by $G^{(p)}_1(t,u,v,z)$.

Note that
\begin{eqnarray*}
G^{(p)}(t,u,v,z,x) &=&  \frac{1}{1-tz}+ \sum_{r \geq 1} G^{(p)}_r(t,u,v,z)x^r \\
&=& \frac{1}{1-tz}+ \sum_{r \geq 1} t^{r-1} z^{r-1} G^{(p)}_1(t,u,v,z)  x^r \\
&=& \frac{1}{1-tz}+ \frac{x}{1-tzx} G^{(p)}_1(t,u,v,z).
\end{eqnarray*}
Thus we have the following theorem.
\begin{theorem}\label{mainres}
\begin{eqnarray}
&&G^{(p)}(t,u,v,z,x) = \frac{1}{1-tz}+ \frac{x}{1-tzx} G^{(p)}_1(t,u,v,z).
\end{eqnarray}
\end{theorem}

\section{Specializations of our general results}\label{lenzeros}

In this section, we shall compute the generating function  for 
$p$-ascent sequences
by length and the number of zeros.

For $n \geq 1$,
let $H^{(p)}_{a,b,\ell, n}$ denote the number of $p$-ascent sequences of
length $n$ with $a$ ascents and $b$ zeros which have last letter
$\ell$. Then we first wish to compute
\begin{equation}\label{H1}
H^{(p)}(t,u,v,z) = \sum_{n \geq 1,\ a,b,\ell \geq 0}
H^{(p)}_{a,b,\ell, n} u^a z^b v^\ell t^n.
\end{equation}

Using the same reasoning as in the previous section, we see
that
\begin{eqnarray*}
H^{(p)}(t,u,v,z) &=& tz+\sum_{\stackrel{a,b,\ell \geq 0}{n \geq 1}}
H^{(p)}_{a,b,\ell,n}t^{n+1} \left(u^av^0z^{b+1}+\sum_{i=1}^{\ell}u^av^iz^b+\sum_{i=\ell+1}^{a+p}u^{a+1}v^iz^b\right)\\
&=& tz+t\sum_{\stackrel{a,b,\ell\geq
0}{n \geq r+1}}H_{a,b,\ell,n}t^nu^az^b\left(z+\frac{v^{\ell+1}-v}{v-1}+u\frac{v^{a+p+1}-v^{\ell+1}}{v-1}\right)\\
&=&
tz+tzH^{(p)}(t,u,1,z) + \frac{tv}{v-1}
\left(H^{(p)}(t,u,v,z) -H^{(p)}(t,u,1,z)\right) + \\
&& \  \frac{tuv}{v-1} \left(H^{(p)}(t,uv,1,z) -H^{(p)}(t,u,v,z)\right).
\end{eqnarray*}
Solving for $H^{(p)}(t,u,v,z)$, we see that we have the following 
lemma.

\begin{lemma}\label{lemH1}
\begin{multline}\label{recH1}
(v\delta_1-1)H^{(p)}(t,u,v,z) = \\
 (v-1)tz+t (z(v-1)-v)H^{(p)}(t,u,1,z) 
+tuv^{p+1}H^{(p)}(t,uv,1,z).
\end{multline}
\end{lemma}
Again, the substitution $v = \frac{1}{\delta_1}$ eliminates the left-hand side of (\ref{recH1}). We
can then solve for $H^{(p)}(u,1,z,t)$ to obtain the recursion
\begin{equation}\label{recH2}
H^{(p)}(t,u,1,z) = \frac{(1-\delta_1)z}{\gamma_1} + \frac{u}{ \gamma_1 
\delta_1^p}
H^{(p)}(t,\frac{u}{\delta_1},1,z).
\end{equation}

We can iterate the recursion (\ref{recH2}) in the same manner 
as we iterated the recursion (\ref{relation2}) in the previous 
section to prove that 
\begin{equation}
H^{(p)}(t,u,1,z) = \sum_{n \geq 0} 
\frac{(\delta_n-\delta_{n+1})z u^n}{\gamma_1 \cdots \gamma_{n+1} 
\delta_n^p}.
\end{equation}
Notice that for all $n \geq 0$, 
\begin{eqnarray*}
\delta_n - \delta_{n+1} &=& (u-(1-t)^n(u-1)) -  (u-(1-t)^{n+1}(u-1)) \\
&=& -(1-t)^n(u-1)(1-(1-t) \\
&=&(1-u)t(1-t)^n.
\end{eqnarray*}
Thus, as a power series in $u$, we can conclude the following.
\begin{theorem}\label{thmH1}
\begin{equation}\label{serH1}
H^{(p)}(t,u,1,z)=\sum_{n=0}^\infty \frac{zt(1-u)u^n(1-t)^n}{\delta_n^p \prod_{i=1}^{n+1} \gamma_i}.
\end{equation}
\end{theorem}

We would like to set $u=1$ in the power series
$\sum_{s=0}^\infty
\frac{zt(1-u)u^s (1-t)^s}{\delta_s \prod_{i=1}^{s+1} \gamma_i}$,
but the factor $(1-u)$ in the series does not allow us to do that
in this form. Thus our next step is to rewrite the series
in a form where it is obvious that we can set $u=1$ in the series.
To that end, observe that for $k \geq 1$,
\begin{equation*}
\delta_k = u -(1-t)^k(u-1) = 1+u-1 -(1-t)^k(u-1) = 1 -((1-t)^k-1)(u-1),
\end{equation*}
so that by Newton's binomial theorem, 
\begin{eqnarray}\label{H4}
\frac{1}{\delta_k^p} &=& \left(\frac{1}{1-(u-1)((1-t)^k-1)}\right)^p 
\nonumber \\
&=&  
\sum_{n=0}^\infty \binom{p-1+n}{n} ((u-1)((1-t)^k-1))^n \nonumber \\
&=& \sum_{n=0}^\infty \binom{p-1+n}{n} (u-1)^n 
\left( \sum_{m=0}^n (-1)^{n-m} \binom{n}{m} (1-t)^{km}\right).
\end{eqnarray}

Substituting (\ref{H4}) into (\ref{serH1}), we see that
\begin{eqnarray*}
&&H^{(p)}(t,u,1,z) =\\
&& \frac{zt(1-u)}{\gamma_1} +
\sum_{k \geq 1} \frac{zt(1-u)u^k (1-t)^k}{\prod_{i=1}^{k+1} \gamma_i}
\sum_{n\geq 0} \binom{p-1+n}{n} (u-1)^n \sum_{m=0}^n  (-1)^{n-m}\binom{n}{m}(1-t)^{km} =\\
&&  \frac{zt(1-u)}{\gamma_1} + \sum_{n \geq 0} \sum_{m=0}^n
 (-1)^{n-m-1}\binom{n}{m}(u-1)^{n-m}zt \sum_{k \geq 1}
\frac{(u-1)^{m+1}u^k(1-t)^{k(m+1)}}{\prod_{i=1}^{k+1} \gamma_i} =\\
&&  \frac{zt(1-u)}{\gamma_1} + \sum_{n \geq 0} \binom{p-1+n}{n} \sum_{m=0}^n
 (-1)^{n-m-1}\binom{n}{m}(u-1)^{n-m}\frac{zt}{(1-zt)^{m+1}} \times \\
&& \ \ \ \sum_{k \geq 1}
\frac{(u-1)^{m+1}(1-zt)^{m+1}u^k(1-t)^{k(m+1)}}{\prod_{i=1}^{k+1} \gamma_i}.
\end{eqnarray*}

In \cite{KR}, we have proved the following lemma.
\begin{lemma}\label{psilem}
\begin{eqnarray}\label{psiform}
\psi_{m+1}(u) &=&  \sum_{k \geq 0}
\frac{(u-1)^{m+1}(1-zt)^{m+1}u^k(1-t)^{k(m+1)}}{\prod_{i=1}^{k+1} \gamma_i}
\nonumber \\
&=& - \sum_{j=0}^m (u-1)^j (1-zt)^j u^{m-j}
\prod_{i=j+1}^m (1 -((1-t)^i).
\end{eqnarray}
\end{lemma}

It thus follows that

\begin{eqnarray*}
H^{(p)}(t,u,1,z) &=&   \frac{zt(1-u)}{\gamma_1} + \sum_{n \geq 0} 
\binom{p-1+n}{n} \sum_{m=0}^n
 (-1)^{n-m-1}\binom{n}{m}(u-1)^{n-m}\frac{zt}{(1-zt)^{m+1}} \times \\
&& \left(-\frac{(u-1)^{m+1}(1-zt)^{m+1}}{\gamma_1} - \sum_{j=0}^m (u-1)^j (1-zt)^j u^{m-j}
\prod_{i=j+1}^m (1 -(1-t)^i)\right).
\end{eqnarray*}
There is no problem in setting $u=1$ in this expression to obtain that
\begin{equation}
H^{(p)}(t,1,1,z) = \sum_{n\geq 0} \binom{p-1+n}{n} \frac{zt}{(1-zt)^{n+1}} \prod_{i=1}^n
(1-(1-t)^i).
\end{equation}

Clearly, our definitions ensure that
$1+H(t,1,1,z) = A^{(p)}(t,z)$ as defined in the introduction so
that we have the following theorem.

\begin{theorem}\label{mainzeros} For all $p \geq 1$,
\begin{equation}\label{eqzeros} 
A^{(p)}(t,z) = \sum_{w \in Asc(p)} t^{|w|}z^{|w|_0} =
1+ \sum_{n \geq 0} \binom{p-1+n}{n} \frac{zt}{(1-zt)^{n+1}} \prod_{i=1}^n (1-(1-t)^i).
\end{equation}
\end{theorem}

The case $p=1$ in Theorem~\ref{mainzeros} gives exactly the same formula for 
$A^{(1)}(t,z)$ as that derived in \cite{KR}, which should be the case. 
We also note that the authors conjectured in \cite{KR} that 

\begin{equation}\label{equality-zeros}
1+ \sum_{k=0}^\infty \frac{zt}{(1-zt)^{k+1}} \prod_{i=1}^k (1-((1-t)^i) 
= 1 + \sum_{m=1}^\infty \prod_{i=1}^m (1-(1-t)^{i-1}(1-zt)).
\end{equation}

This was proved independently by Jel\'inek \cite{Jel}, Levande 
\cite {levande}, and Yan \cite{Yan}. It would be interesting 
to find an analogue of this relation for $p > 1$.

We have used Mathematica to compute the first few terms
of $A^{(p)}(t,z)$ for $p=2,3,4$:
\begin{eqnarray*}
&&A^{(2)}(t,z) = 1+ z t+\left(2z+z^2\right) t^2+\left(6 z+4 z^2+z^3\right) t^3+\left(21 z+18 z^2+6 z^3+z^4\right) t^4 +\\
&& \left(84 z+87 z^2+36 z^3+8 z^4+z^5\right) t^5+\left(380 z+456
z^2+222 z^3+60 z^4+10 z^5+z^6\right) t^6+O[t]^7.
\end{eqnarray*}

\begin{eqnarray*}
&&A^{(3)}(t,z) = 1+ z t+\left(3z+z^2\right) t^2+\left(12 z+6 z^2+z^3\right) t^3+\left(54 z+36 z^2+9 z^3+z^4\right) t^4 +\\
&&\left(270 z+222 z^2+72 z^3+12 z^4+z^5\right) t^5+\\
&&\left(1490 z+140
z^2+564 z^3+120 z^4+15 z^5+z^6\right) t^6+O[t]^7.
\end{eqnarray*}

\begin{eqnarray*}
&&A^{(4)}(t,z) = 1+ z t+\left(4z+z^2\right) t^2+\left(20 z+8 z^2+z^3\right) t^3+\left(110 z+60 z^2+12 z^3+z^4\right) t^4 +\\
&&\left(660 z+450 z^2+90 z^3+16 z^4+z^5\right) t^5+\\
&&\left(4300 z+3480
z^2+1140 z^3+200 z^4+20 z^5+z^6\right) t^6+O[t]^7.
\end{eqnarray*}

Next we can use the same techniques as in \cite{DKRS} 
to find the generating function for the number of {\em primitive 
$p$-ascent sequences}. That is, 
let $r_{n,p}$ denote the number of $p$-ascent sequences $a$ 
of length $n$ such that $a$ has no consecutive repeated letters 
and $a_{n,p}$ denote the number of $p$-ascent sequences $a$ 
of length $n$.

If 
\begin{eqnarray*}
R^{(p)}(t) &=&  1+ \sum_{n\geq 1} r_{n,p} t^n \ \mbox{and} \\
A^{(p)}(t) &=&  1+ \sum_{n\geq 1} a_{n,p} t^n,
\end{eqnarray*}
then it is easy to see that 
\begin{equation}\label{primitive1}
A^{(p)}(t) = A^{(p)}(t,1) = R^{(p)}\left(\frac{t}{1-t}\right)= R^{(p)}(t+t^2+\cdots),
\end{equation}
since each element in a primitive $p$-ascent sequence can be repeated any specified number of times.

Setting $x = \frac{t}{1-t}$ so that $t = \frac{x}{1+x}$, we 
see that (\ref{primitive1}) implies that 
\begin{equation}\label{primitive2}
R^{(p)}(x) = A^{(p)}\left(\frac{x}{1+x}\right).
\end{equation}
But by (\ref{eqzeros}), we know that 
$$A^{(p)}(t) = 1+ \sum_{n=0}^\infty \binom{p-1+n}{n} 
\frac{t}{(1-t)^{n+1}} \prod_{i=1}^n (1-(1-t)^i).$$
Hence, 
\begin{eqnarray*}
R^{(p)}(x) &=&  
1+ \sum_{n=0}^\infty \binom{p-1+n}{n} 
\frac{\frac{x}{1+x}}{(1-\frac{x}{1+x})^{n+1}} \prod_{i=1}^n \left(1-\left(1-\frac{x}{1+x}\right)^i\right) \\
&=& 1+ x\sum_{n=0}^\infty \binom{p-1+n}{n}  (1+x)^n 
\prod_{i=1}^n \left(1-\left(\frac{1}{1+x}\right)^i\right).
\end{eqnarray*}

Thus, we have the following theorem. 
\begin{theorem}\label{thm:primitive} For all $p \geq 1$, 
\begin{equation}\label{primitive3}
R^{(p)}(t) = 
1+ t\sum_{n=0}^\infty \binom{p-1+n}{n} 
(1+t)^n \prod_{i=1}^n \left(1-\left(\frac{1}{1+t}\right)^i\right).
\end{equation}
\end{theorem}

For example, we have computed that 
\begin{eqnarray*}
&&R^{(2)}(t) = 1+t +2t^2 + 6t^3+21t^4+87t^5 + 413 t^6 +2213 t^7 
+ 13205 t^8 + 86828 t^9 +O[t]^{10},\\
&&R^{(3)}(t) = 1+t +3t^2 + 12t^3+54t^4+276t^5 + 1574 t^6 +9916 t^7 
+ 68394 t^8 + 512671 t^9 +O[t]^{10}, \\
&&R^{(4)}(t) = 1+t +4t^2 + 20t^3+110t^4+670t^5 + 4470 t^6 +32440 t^7 
+ 254490 t^8 + 2146525 t^9 +\\
&&O[t]^{10}.
\end{eqnarray*}

We note that by (\ref{equality-zeros}), in the case of $z=1$, we have that 
$$A^{(1)}(t) = 1 + \sum_{m=1}^\infty \prod_{i=1}^m (1-(1-t)^i)=1+ \sum_{n \geq 0} \frac{t}{(1-t)^{n+1}} \prod_{i=1}^n (1-(1-t)^i).$$

This leads to two forms for  
$R^{(1)}(t) = A^{(1)}\left( \frac{t}{1+t} \right)$. That is, it 
follows that  
\begin{eqnarray*}
R^{(1)}(t) &=& 1+ t \sum_{n=0}^\infty (1+t)^n \prod_{i=1}^n 
\left( 1 - \left( \frac{1}{1+t}\right)^i \right) = 1 + \sum_{m=1}^\infty \prod_{i=1}^m \left(1-\frac{1}{(1+t)^i}\right),
\end{eqnarray*}
where the second formula was derived in \cite{DKRS}. It would 
be interesting to find an analogue of this equality for 
$p > 1$. 

Finally if we replace $t$ by $t+ t^2 + \cdots + t^k = t\frac{(t^k-1)}{t-1}$ 
in (\ref{primitive3}), then we can obtain the generating function for the 
number of $p$-ascent sequences $a$ such that the maximum length of a 
consecutive sequence of repeated letters is less than or equal to $k$: 

\begin{equation}\label{max-k-rep}
1+ t\frac{t^k-1}{t-1}\sum_{n=0}^\infty \binom{p-1+n}{n}\left(\frac{t^{k+1}-1}{t-1}\right)^n \prod_{i=1}^n \left(1-\left(\frac{t-1}{t^{k+1}-1}\right)^i\right).
\end{equation}

\section{Pattern avoidance in $p$-ascent sequences}\label{pattern-avoidance}

In this section, we shall prove some simple results about 
pattern avoidance in $p$-ascent sequences thus extending the studies initiated in \cite{DS} for ascent sequences.

We begin by considering patterns of length 2.  There are three 
such patterns, 00, 01, and 10.  Recall that $a_{n,p,u}$ (resp., $r_{n,p,u}$) is the number of (resp., primitive) $p$-ascent sequences of length $n$ that avoid a pattern $u$.

\subsection{01-avoiding $p$-ascent sequences}

The only $p$-ascent sequences 
that avoid 01 are the sequences that consist of all  zeros so 
that $a_{n,p,01} =1$ for all $n,p \geq 1$ and 
$r_{n,p,01}$ equals 1 if $n=1$ and 0 otherwise.

\subsection{10-avoiding $p$-ascent sequences}

Let us consider $r_{n,p,10}$. In this case, we are looking 
for $p$-ascent sequences which avoid 10 and have no repeated 
letters.  It is clear that any such a sequence $a$ must be of the 
form $a =a_1\ldots a_n$, where $0=a_1 < a_2 < \cdots < a_n$. 
For each $1 \leq i \leq n$, the word $a_1 \ldots a_i$ has $i-1$ 
ascents so that $a_{i+1} \leq i-1+p$. It follows 
that $r_{n,p,10}$ counts all words $a_1 a_2 \ldots a_n$, where $0=a_1 < a_2 < \cdots < a_n \leq p+n-2$.  Hence 
\begin{equation}
r_{n,p,10} = \binom{p+n-2}{n-1} .
\end{equation}

Note that it follows from 
Newton's Binomial Theorem that 
\begin{eqnarray}
R^{(p)}_{10}(t) &=& 1+ \sum_{n \geq 1} r_{n,p,10} t^n \nonumber \\
&=& 1 + \sum_{n \geq 1} \binom{p-1+n-1}{n-1} t^n \nonumber \\
&=& 1+ \frac{t}{(1-t)^p}. 
\end{eqnarray}

It is easy to see that the $p$-ascent sequences counted 
by $a_{n,p,10}$ arise by taking a  sequence $d_1 \ldots d_s$ counted 
by $r_{s,p,10}$ for some $s \leq n$ and replacing 
each letter $d_i$ by one or more copies so that the resulting 
word is of length $n$.  The number of ways to do this for a given 
$d_1 \ldots d_s$ is the number of solutions to 
$b_1+ \cdots +b_s =n$, where $b_i \geq 1$, which is 
$\binom{n-1}{s-1}$. Thus 
\begin{equation}
a_{n,p,10}= \sum_{s=1}^n \binom{n-1}{s} r_{s,p,10} = 
 \sum_{s=1}^n \binom{n-1}{s-1} \binom{p+s-2}{s-1}=\sum_{s=0}^{n-1} \binom{n-1}{s} \binom{p+s-1}{s}.
\end{equation}

It also follows that 

\begin{eqnarray}
A^{(p)}_{10}(t) &=&  1+ \sum_{n \geq 1} a_{n,p,10} t^n \nonumber \\
&=& R^{(p)}_{10}\left( \frac{t}{1-t} \right ) \nonumber \\
&=& 1+\frac{t}{1-t} \frac{1}{(1-\frac{t}{1-t})^p} \nonumber \\
&=& 1+\frac{t(1-t)^{p-1}}{(1-2t)^p}. 
\end{eqnarray}

We note that the sequence $(a_{n,2,10})_{n \geq 1}$ starts out 
$1,3,8,20,48,112,256, \ldots$ and this is 
the sequence  A001792 in the OEIS~\cite{oeis} which has  
many combinatorial interpretations.

\subsection{00-avoiding $p$-ascent sequences}

Next,  consider avoiding the pattern 00. If a $p$-ascent sequence 
$a=a_1 \ldots a_n$ avoids 00, then all its  elements 
must be distinct.  Note that for each $2 \leq i \leq n$, 
$a_1 \ldots a_{i-1}$ can have at most $i-2$ ascents so 
that $a_i \leq p+i-2$.  Let $\max(a)$ denote the maximum 
of $\{a_1, \ldots, a_n\}$.  If $a$ avoids 00, then 
by the pigeon hole principle, it must be the case 
that $\max(a) \geq n-1$.  Thus, if $a$ avoids 00, 
then $n-1 \leq \max(a) \leq n+p-2$.

Now consider $2$-ascent sequences that avoid 00.  Suppose that 
$a=a_1 \ldots a_n$ is a 2-ascent sequence which avoids 00. Then we know that  
$\max(a) \in\{n-1,n\}$.  If $\max(a) =n$, $a$ must be strictly increasing 
and there must be some smallest $k\geq 1$ such that $a_k=k$,  
In such a situation, 
it is easy to see that $a$ must be of the form 
$0,1,\ldots, k-2,k,k+1, \ldots n$.  Thus there are 
$n-1$ 2-ascent sequences $a$ of length $n$ such that $a$ avoids 
00 and $\max(a) =n$.  

Next, suppose that $a =a_1 \ldots a_n$ is a 2-ascent sequence that 
avoids 00 and $\max(a) =n-1$.  Then there are two cases. Namely, 
it could be that there is no $k \leq n$ such that $a_k =k$. In that 
case, $a$ is the increasing sequence $a = 012 \ldots (n-1)$. 
Otherwise, let $j$ equal the smallest $i$ such that 
$a_i=i$.  Then $a$ must be strictly increasing up to $a_j$ so 
that $a$ starts out $012 \ldots (j-2) j$.  Since $\max(a) = n-1$, 
it follows that $\{a_1, \ldots, a_n\} = \{0,1, \ldots, n-1\}$ so 
that there must be some $j < k \leq n$ such that 
$a_k =j-1$. In that case, $a_{k-1} > a_k$ so that $a$ has at least 
one descent.  However, if $\max(a) =n-1$, $a$ can have at most one 
descent.  Thus, once we have placed $j-1$, the remaining elements 
must be placed in increasing order.  It is then easy to check that 
no matter where we place $j-1$ after position $j$, the resulting 
sequence will be a 2-ascent sequence. It follows 
that the number of 2-ascent sequences which avoid 
00 and have one descent is $\sum_{j=1}{n-1} (n-j) = \binom{n-1}{2}$.

It is easy to check that wherever we place 
$j-1$, $j-1$ will cause a descent and there can be at most 
one descent in $a$.  Thus, we have the following theorem. 

\begin{theorem} For all $n \geq 1$, 
\begin{equation}
a_{n,2,00} = n-1 + 1 +\binom{n-1}{2} = 1+ \binom{n}{2}.
\end{equation}
\end{theorem}

We computed that the sequence $(a_{n,3,00})_{n\geq 1}$ starts 
out  
$$1,3,9,24,57,122,239,435,745,1213,1893,2850, \ldots.$$
This is the sequence A089830 in the OEIS~\cite{oeis}, 
whose generating function is $$\frac{1-3x+6x^2-5x^3+3x^4-x^5}{(1-x)^6}.$$

In this case, if $a=a_1 \ldots a_n$ is a 3-ascent sequence which avoids 00, 
then we know that $n-1 \leq \max(a)\leq n+1$.  We shall  prove that  
$$\sum_{n \geq 1} a_{n,3,00} x^n = 
\frac{x(1-3x+6x^2-5x^3+3x^4-x^5)}{(1-x)^6}$$ by classifying  
the 3-ascent sequences $a$ which avoid 00 by the $\max(a)$ and 
$\des(a)$, where $\des(a)$ is the number of {\em descents} in $a$, that is, the number of elements followed by smaller elements. \\
\ \\
{\bf Case 1}. $\des(a) =0$. \\
Suppose that $a = a_1 \ldots a_n$ is an increasing 3-ascent 
sequence that avoids 00.  Now, if $\max(a) =n-1$, 
then $a = 012 \ldots (n-1)$.  If $\max(a) = n$, 
then exactly one element from $[n]=\{1, \ldots, n-1\}$ does not 
appear in $a$.  If $i$ does not appear in $a$, 
then $a = 01 \ldots (i-1)(i+1) (i+2) \ldots n$, which is 
a 3-ascent sequence. Thus, there are $n-1$ increasing 3-ascent sequences 
whose maximum is $n$. Finally, if $\max(a) =n+1$, then  
two elements from $[n]$ do not appear in $a$. Again, it is easy 
to check that no matter which two elements from $[n]$ we leave out, 
the resulting increasing sequence will be a 3-ascent sequence. Thus, 
there are $\binom{n}{2}$ increasing 3-ascent sequences whose 
maximum is $n+1$. Therefore, the total number of increasing 3-ascents 
sequences of length $n$ is $1+(n-1)+\binom{n}{2} = \binom{n+1}{2}$. \\
\ \\
{\bf Case 2.} $\des(a) = 1$. \\
In this case, if $a =a_1 \ldots a_n$ is a 3-ascent sequence such 
that $\des(a) =1$ and $a$ avoids 00, then $\max(a) \in \{n-1,n\}$. 
Suppose that $a_j > a_{j-1}$.  Then we have two subcases. \\
\ \\
{\bf Subcase 2.1.} $a_j =j+1$.  \\
In this case,  there must be two elements $1 \leq u < v \leq j$, which 
do not appear in $a_1 \ldots a_j$.  Clearly, we have 
$\binom{j}{2}$ ways to pick $u$ and $v$. We then 
have three subcases.\\
\ \\
{\bf Subcase 2.1.1.} Both $u$ and $v$ appear in $a$. In this case, 
$a$ must start out $a_1 \ldots a_j u v$ so that 
$a_{j+3} \ldots a_n$ must be an increasing sequence 
from $[n]-[j+1]$ of length $n-j-2$.  Clearly, there 
are $n-j-1$ such subsequences and it is easy to check 
that we can attach any such subsequence at the end 
of the sequence $a_1 \ldots a_j u v$ to obtain 
a 3-ascent sequence avoiding 00. \\
\ \\
{\bf Subcase 2.1.2.} $u$ appears in $a$, but $v$ does not appear in $a$.\\
In this case, $a$ must be of the form $a_1 \ldots a_j u \gamma$, where 
$\gamma$ is the increasing sequence $(j+2) (j+3) \ldots n$. \\ 
\ \\
{\bf Subcase 2.1.3.} $v$ appears in $a$, but $u$ does not appear in $a$.\\
In this case, $a$ must be of the form $a_1 \ldots a_j v \gamma$, where 
$\gamma$ is the increasing sequence $(j+2) (j+3) \ldots n$. \\
\ \\
It follows that the number of 3-ascent sequences counted in 
Case 2.1 is $\sum_{j=2}^{n-1}\binom{j}{2}(n-j+1)$. One can 
verify by Mathematica that 
$$\sum_{j=2}^{n-1}\binom{j}{2}(n-j+1) = \binom{n}{3} + \binom{n+1}{4}.$$
\ \\
{\bf Case 2.2.}  $a_j =j$. \\
In this case, there is one element $u$ in $[j]$ which does 
not appear in $a_1 \ldots a_j$, so that the sequence 
must start out $a_1 \ldots a_j u$.  The rest of 
the sequence must be the increasing rearrangement of 
$\{j+1, \ldots, n\}- \{v\}$ for some $v \in \{j+1, \ldots, n\}$. 
Thus, we have $j-1$ choices for $u$ and $n -j$ choices for $v$. 
Hence the number of 3-ascent sequences in Case 2.2 is 
$\sum_{j=2}^{n-1}(j-1)(n-j)$. One can check by Mathematica that 
$\sum_{j=2}^{n-1}(j-1)(n-j) = \binom{n}{3}$.

Thus, the number of 3-ascent sequences with one descent, which 
avoid 00 is $2\binom{n}{3} + \binom{n+1}{4}$.\\
\ \\
{\bf Case 3} $\des(a) =2$.\\
In this case, it must be that $\max(a) =n-1$, so that 
$a$ must contain all the elements in the sequence $0,1, \ldots ,n-1$. 
Now, suppose that the first descent of $a$ occurs at position $j$. 
Then we have two cases. \\
\ \\
{\bf Case 3.1} $a_j =j$. \\
In this case, there must be $u$, $1 \leq u \leq j-1$, which does 
not appear in $a_1 \ldots a_j$ and $a_{j+1}=u$. We have $j-1$ choices 
for $u$.  The sequence $a_{j+2} \ldots a_n$ must 
be a rearrangement of $(j+1) (j+2) \ldots (n-1)$, which has one 
descent. The bottom element of the descent pair that 
occurs in $a_{j+2} \ldots a_n$ must equal $s$ for some 
$j+1 \leq s \leq n-2$ and the top element of the descent 
must equal $t$, where $s+1 \leq t \leq n-1$.  It is easy to check 
that any choice of $s$ and $t$ will yield a 3-ascent sequence,  
so that the number of choices for the sequence $a_{j+2} \ldots a_n$ 
is 
\begin{eqnarray*}
\sum_{s =(j+1)}^{n-2} n-1-s &=& \sum_{r=1}^{n-2-j} n-1-(r+j) \\
&=& \sum_{r=1}^{n-2-j} n-1-j -r  = \binom{n-1-j}{2}.
\end{eqnarray*}

It follows that the number of 3-ascent sequences in Case 3.1 is 
$\sum_{j=2}^{n-2} (j-1)  \binom{n-1-j}{2}$, which can be shown 
by Mathematica to be equal to $\binom{n-1}{4}$. \\
\ \\
{\bf Case 3.2} $a_j =j+1$. \\
In this, there must be two elements $1 \leq u \leq v \leq j$ that 
do note appear in $a_1 \ldots a_j$. We have $\binom{j}{2}$ ways 
to choose $u$ and $v$. We then have two subcases. \\
\ \\
{\bf Case 3.2.1} $a_{j+1} =v$.  \\
In this case, our sequences start out $a_1 \ldots a_j = (j+1) v$ 
and where every $u$ occurs in the sequence $a_{j+2} \ldots a_n$, it 
will cause a second descent so that there are $n-j-1$ choices in 
this case. \\
\ \\
{\bf Case 3.2.2} $a_{j+1} = u$. \\
In this case, the sequence $a_{j+2} \ldots a_n$ consists 
of the sequence $v (j+2) (j+3) \ldots (n-1)$ and we can argue 
as we did in Case 3.1 that there are $\binom{n-j-1}{2}$ choices 
for the sequence $a_{j+2} \ldots a_n$.\\
\ \\
It follows that the total number of choices for the sequence 
$a_{j+1} \ldots a_n$ in Case 3 is $n-j-1+\binom{n-j-1}{2} = \binom{n-j}{2}$. 
Thus the total number of choices in Case 3.2 is  
$$\sum_{j=1}^{n-2} \binom{j}{2} \binom{n-j}{2} = \binom{n+1}{5}.$$
Note that the last equality can be checked by Mathematica.

Putting all the cases together, we see that the number of 
3-ascent sequences of length $n$, which avoid 00 is equal to 
$$\binom{n+1}{2} + 2 \binom{n}{3} + \binom{n+1}{4} + \binom{n-1}{4} + 
\binom{n+1}{5}.$$
Using the fact that $\binom{n+1}{4}+ \binom{n+1}{5} = 
\binom{n+2}{5}$, we see that we have the following theorem. 

\begin{theorem} For all $n \geq 1$, 
$$a_{n,3,00} = \binom{n+1}{2} + 2 \binom{n}{3} + \binom{n-1}{4} + 
\binom{n+2}{5}.$$
\end{theorem}

Note that it follows from Newton's binomial theorem 
that 
\begin{eqnarray*}
\sum_{n \geq 1} \binom{n+1}{2} x^n &=& \frac{x}{(1-x)^3}, \\
\sum_{n \geq 1} 2 \binom{n}{3} x^n &=& \frac{2x^3}{(1-x)^4}, \\
\sum_{n \geq 1} \binom{n-1}{4} x^n &=& \frac{x^5}{(1-x)^5}, \ \mbox{and} \\
\sum_{n \geq 1} \binom{n+2}{5} x^n &=& \frac{x^3}{(1-x)^6}.
\end{eqnarray*}
Adding these series together and simplifying, we have the following 
theorem. 
\begin{theorem} The generating function
$$\sum_{n \geq 1} a_{n,3,00} x^n = \frac{x(1-3x+6x^2-5x^3+3x^4-x^5)}{(1-x)^6}.
$$
\end{theorem}

We note that Burstein and Mansour \cite{BM} 
gave a combinatorial interpretation to the $n$-th element in 
sequence A089830 as the 
number of words $w =w_1 \ldots w_{n-1} \in \{1,2,3\}^*$, which 
avoid the vincular pattern 
21-2 (also denoted in the literature $\underline{21}2$; see~\cite{Kit}).  That is, there are no subsequences of the form $w_i w_{i+1} w_j$ 
in $w$ such that $i+1 <j$ and $w_i = w_j > w_{i+1}$.  We ask the question 
whether one can construct a simple bijection between such words 
and the set of 3-ascent sequences of length $n$, which avoid 00.

We note that the sequence $(a_{n,4,00})_{n \geq 1}$ starts  
out $1,4,16,58,190,564,1526,3794 \ldots$.  This sequence 
does appear in the OEIS.

\subsection{012-avoiding $p$-ascent sequences}

Now suppose that $a =a_1 \ldots a_n$ is a $p$-ascent sequence 
such that $a$ avoids 012.  The first thing to observe is that if 
$a_i =1$ for some $i$, then since $a_1=0$, it must be the case 
that $a_j \in \{0,1\}$ for all $j \geq i$. The second 
thing to observe is that $a_i \leq p$ for all $i$.  
That is, the only way that $a$ can have an element $a_k > p$ is if 
$a_1 \ldots a_{k-1}$ has at least $a_k-p$ ascents. Since 
the first ascent in a $p$-ascent sequence must be of one of 
the forms $01, 02, \ldots, 0p$, such an ascent sequence 
would not avoid 012. \\

\noindent
{\bf 2-ascent sequences.} Now, suppose that $a =a_1 \ldots a_n$ is a $2$-ascent sequence  
such that $a$ avoids 012. If $a$ has no 1s, then 
$a_i \in \{0,2\}$ for all $i \geq 2$, so that 
there are $2^{n-1}$ such 2-ascent sequences. If $a$ contains a 
1, then let $k$ be the smallest $j$ such that $a_j$ equals 1. 
It then follows that $a_i \in \{0,2\}$ for $2 \leq i < k$ and 
$a_j \in \{0,1\}$ for $k < j \leq n$. Thus, there 
are $2^{n-2}$ such 2-ascent sequences, so that the number of 
$2$-ascent sequences that avoid 012 and contain a 1 is 
$(n-1)2^{n-2}$. Hence, for $n \geq 1$,  
\begin{equation}
a_{n,2,012} = 2^{n-1} + (n-1)2^{n-2} = (n+1)2^{n-2}.
\end{equation}

We note that the sequence $(a_{n,2,012})_{n \geq 1}$ starts out 
$1,3,8,20,48,112,256, \ldots$, and this is, again, as in the case of  $(a_{n,2,10})_{n \geq 1}$, 
the sequence  A001792 in the OEIS~\cite{oeis}. Next, we will explain this fact combinatorially.

It is easy to see that each sequence counted by $(a_{n,2,012})_{n \geq 1}$ can be obtained by taking a number of 2s (maybe none) followed by a number of 1s, and placing any number of 0s (maybe none) between these 1s and 2s making sure that the total length of the sequence is $n$, and this sequence begins with a 0. On the other hand, it is also straightforward to see that sequences counted by  $(a_{n,2,10})_{n \geq 1}$ are of two types: they are either of the form
\begin{equation}\label{form1-seq}
\underbrace{0\ldots 0}_{i_0\geq 1}\underbrace{1\ldots 1}_{i_1\geq 1}\underbrace{2\ldots 2}_{i_2\geq 1}\ldots,
\end{equation}
or of the form
\begin{equation}\label{form2-seq}
\underbrace{0\ldots 0}_{i_0\geq 1}\underbrace{1\ldots 1}_{i_1\geq 1}\underbrace{2\ldots 2}_{i_2\geq 1}\ldots\underbrace{a\ldots a}_{i_a\geq 1}\underbrace{(a+2)\ldots (a+2)}_{i_{a+2}\geq 1}\underbrace{(a+3)\ldots (a+3)}_{i_{a+3}\geq 1}\underbrace{(a+4)\ldots (a+4)}_{i_{a+4}\geq 1}\ldots,
\end{equation}
where $a\geq 0$ exists. A bijection between the classes of sequences is given by turning sequences of the form (\ref{form1-seq}) into 
$$\underbrace{0\ldots 0}_{i_0}2\underbrace{0\ldots 0}_{i_1-1}2\underbrace{0\ldots 0}_{i_2- 1}\ldots,$$
and the sequences of the form (\ref{form2-seq}) into  
$$\underbrace{0\ldots 0}_{i_0}2\underbrace{0\ldots 0}_{i_1- 1}2\underbrace{0\ldots 0}_{i_2- 1}\ldots2\underbrace{0\ldots 0}_{i_a- 1}1\underbrace{0\ldots 0}_{i_{a+2}-1}1\underbrace{0\ldots 0}_{i_{a+3}- 1}1\underbrace{0\ldots 0}_{i_{a+4}- 1}\ldots.$$

\noindent
{\bf 3-ascent sequences.} Now, suppose that $a =a_1 \ldots a_n$ is a $3$-ascent sequence  
such that $a$ avoids 012. If $a$ has no 1s, then 
$a_i \in \{0,2,3\}$ for all $i \geq 2$. 
It is then easy to see that if $b_1 \ldots b_{n}$ is the 
sequence that arises from $a_1 \ldots a_{n}$ by replacing 
each 2 by a 1 and each 3 by a 2, then $b$ is a 2-ascent 
sequence that avoids 012. Thus, there are 
$(n+1) 2^{n-2}$ such sequences. Now, suppose that $a$ contains a 1. 
Then let $k$ be the smallest $j$ such that $a_j$ equals 1. 
It then follows that $a_i \in \{0,2,3\}$ for $2 \leq i < k$ and 
$a_j \in \{0,1\}$ for $k < j \leq n$. 
It is then easy to see that if $b_1 \ldots b_{k-1}$ is the 
sequence that arises from $a_1 \ldots a_{k-1}$ by replacing 
each 2 by a 1 and each 3 by a 2, then $b_1 \ldots b_{k-1}$ is a 2-ascent 
sequence that avoids 012.  Thus, from our argument above, it follows  that there are $k2^{k-3}$ choices for $a_1 \ldots a_{k-1}$ and 
$2^{n-k}$ choices for $a_{k+1} \ldots a_n$.  Therefore, given $k$, 
we have $k2^{n-3}$ choices for $a$.  
Thus,
\begin{eqnarray}
a_{n,3,012} &=&(n+1)2^{n-2} + \sum_{k=2}^n k 2^{n-3} \nonumber \\
&=& 2^{n-3} \left(2n+2+ \sum_{k=2}^n k\right) \nonumber \\
&=&  2^{n-3}\left(2n+2 +\binom{n+1}{2} -1\right) \\
&=& 2^{n-3}\frac{4n+4 +n^2+n -2}{2} = 2^{n-4}(n^2+5n +2). 
\end{eqnarray}

We note that the sequence $(a_{n,3,012})_{n \geq 1}$ starts out 
$1, 4, 13, 38, 104, 272, 688, \ldots$ and this is 
the sequence  A049611 in the OEIS~\cite{oeis} having several combinatorial interpretations. \\

\noindent
{\bf $p$-ascent sequences for an arbitrary $p$.} In general, we can obtain a simple recursion for 
$a_{n,p,012}$.  That is, 
suppose that $a =(a_1, \ldots, a_n)$ is a $p$-ascent sequence 
such that $a$ avoids 012. Now, if $a$ has no 1s, then 
$a_i \in \{0,2,3, \ldots ,p\}$ for all $i \geq 2$. 
It is then easy to see that if $b=(b_1, \ldots, b_{n})$ is the 
sequence that arises from $a$ by replacing 
each $i \geq 2$,  by an $i-1$, then $b$ is a $(p-1)$-ascent 
sequences that avoids 012. Thus, there are 
$a_{n,p-1,012}$ such sequences. Now suppose that $a$ contains a 1. 
Then let $k$ be the smallest $j$ such that $a_j$ equals 1. 
It then follows that $a_i \in \{0,2,3, \ldots, p\}$ for $2 \leq i < k$ and 
$a_j \in \{0,1\}$ for $k < j \leq n$. 
It is then easy to see that if $b_1 \ldots b_{k-1}$ is the 
sequence that arises from $a_1 \ldots a_{k-1}$ by replacing 
each $i \geq 2$ by an $i-1$, then $b_1 \ldots b_{k-1}$ is a 2-ascent 
sequences that avoids 012.  It follows  
that there are $a_{k-1,p-1,012}$ choices for $a_1 \ldots a_{k-1}$ and 
$2^{n-k}$ choices for $a_{k+1} \ldots a_n$.  Thus, given $k$, 
we have $2^{n-k}a_{k-1,p-1,012}$ choices for $a$.  
It follows that 
\begin{equation}
a_{n,p,012} = a_{n,p-1,012} + \sum_{k=2}^n a_{k-1,p-1,012} 2^{n-k}.
\end{equation}

Thus, for example, 
\begin{eqnarray*}
a_{n,4,012} &=& 
2^{n-4}(n^2+5n+2) + \sum_{k=2}^n 2^{k-5}((k-1)^2 + 5(k-1)+2)2^{n-k}\\
&=& 2^{n-5}\left(2n^2 + 10n+4+\sum_{k=1}^{n-1} (k^2 +5k +2)\right) \\
&=& 2^{n-5} (2n^2 +10n+4 +(1/3)(n^3 +6n^2-n-6)) \\
&=& \frac{2^{n-5}}{3}(n^3+12n^2+29n+6).
\end{eqnarray*}
We note that the sequence $(a_{n,4,012})_{n \geq 1}$ starts out 
$$1,5,19,63,192,552,1520,4048,10496,26264, \ldots$$ and this is 
the sequence  A049612 in the OEIS~\cite{oeis}.

\end{document}